\documentclass{amsart}

\usepackage{amssymb}
\usepackage{amsmath}
\usepackage{latexsym} 
\usepackage{amsfonts}
\usepackage{amsthm}
\usepackage{mathrsfs}
\usepackage{verbatim} % needed for \begin{comment} \end{comment}
\usepackage[all,2cell,ps]{xy}
\usepackage{tikz}
\usepackage{caption}
\usepackage{subcaption}
\usepackage{float}
\usepackage{colortbl}
\usepackage[pdftex,colorlinks=true, pdfstartview=FitV, linkcolor=magenta, citecolor=blue, urlcolor=blue]{hyperref}
\usepackage{enumitem}

\setlist[description]{style=nextline,labelwidth=0pt,leftmargin=30pt,itemindent=\dimexpr-20pt-\labelsep\relax}
\makeatletter % Redefinition of Description List Items source: https://tex.stackexchange.com/a/1248/13552
\let\orgdescriptionlabel\descriptionlabel
\renewcommand*{\descriptionlabel}[1]{%
  \let\orglabel\label
  \let\label\@gobble
  \phantomsection
  \edef\@currentlabel{#1}%
  \let\label\orglabel
  \orgdescriptionlabel{#1}%
}
\makeatother

\newcommand{\pure}[1]{\pi(#1)}

\newcommand{\dd}{\mathbf d}
\newcommand{\ee}{\mathbf e}
\newcommand{\bet}{\beta} % notation for betti diagram if context is given
\newcommand{\betti}[1]{\bet{(#1)}} % notation for betti diagram of a ring
\DeclareMathOperator{\concat}{concat}

\newcommand{\A}{a} % sum of a_i's
\newcommand{\elimo}{\mathcal E} % elimination order of a betti diagram if context is given
\newcommand{\elims}{\varepsilon} % elimination size of a betti diagram if context is given
\newcommand{\elimorder}[1]{\elimo (#1)} % elimination order of a betti diagram
\newcommand{\elimsize}[1]{\elims (#1)} % elimination order of a betti diagram

\newcommand{\defi}[1]{\textsf{#1}} % for defined terms

\makeatletter
\def\adots{\mathinner{\mkern1mu\raise\p@
\vbox{\kern7\p@\hbox{.}}\mkern2mu
\raise4\p@\hbox{.}\mkern2mu\raise7\p@\hbox{.}\mkern1mu}}
\makeatother

\newcommand{\twiddle}{{\raise.17ex\hbox{$\scriptstyle\sim$}}}

\newcommand{\kk}{\Bbbk}

\newtheorem{theorem}{Theorem}[section]
\newtheorem{thm}[theorem]{Theorem}
\newtheorem{prop}[theorem]{Proposition}
\newtheorem{cor}[theorem]{Corollary}
\newtheorem{algorithm}[theorem]{Algorithm}
\newtheorem{lemma}[theorem]{Lemma}
\newtheorem{lem}[theorem]{Lemma}

\newtheorem{conjecture}[theorem]{Conjecture}

\theoremstyle{definition}
\newtheorem{definition}[theorem]{Definition}
\newtheorem{dfn}[theorem]{Definition}
\newtheorem{example}[theorem]{Example}
\newtheorem{remark}[theorem]{Remark}

\newcommand{\ds}{\displaystyle}

\newcommand{\ul}{\underline}

\newcommand{\Z}{\mathbb{Z}}
\newcommand{\Q}{\mathbb{Q}}

\newcommand{\Reg}{\mathrm{reg}}
 
\newcommand{\Hom}{\mathrm{Hom}}

\newcommand{\pdim}[1]{\mathrm{pd}(#1)}
\newcommand{\codim}{\mathrm{codim}}

\newcommand{\BS}{Boij-S\"oderberg\ }

\newcommand{\ls}{\leq}

\newcommand{\lst}[2]{#1_1,\dots,#1_{#2}}
\newcommand{\reg}{\mathrm{reg}}

\usetikzlibrary{shapes}

\newcommand\encircle[1]{%
  \tikz[baseline=(X.base)] 
    \node (X) [draw, shape=circle, fill=red!30!white, inner sep=0] {\strut #1};}

\newcommand\enwithdiamond[1]{%
  \tikz[baseline=(X.base)] 
    \node (X) [draw, shape=diamond, fill=blue!30!white, inner sep=0] {\strut #1};}    

\newcommand\ensquare[1]{%
  \tikz[baseline=(X.base)] 
    \node (X) [draw, shape=regular polygon, regular polygon sides=4, fill=green!40!gray, inner sep=-1] {\strut #1};}

\newcommand\entriangle[1]{%
  \tikz[baseline=(X.base)] 
    \node (X) [draw, shape=regular polygon, regular polygon sides=5, fill=brown!30!white, inner sep=-1] {\strut #1};}

\allowdisplaybreaks

\begin{document}
%----------------------------
% Title
%----------------------------

\title[Recursive Decompositions]{Recursive strategy for decomposing Betti tables of complete intersections}

\author[Gibbons]{Courtney R. Gibbons}
\address{Courtney R. Gibbons, Mathematics Department, Hamilton College, 198 College Hill Road, Clinton, NY 13323}
\email{crgibbon@hamilton.edu}

\author[Huben]{Robert Huben}
\address{Robert Huben, Department of Mathematics, University of Nebraska--Lincoln, P.O. Box 880130, Lincoln, NE 68588-0130}
\email{rhuben@huskers.unl.edu}

\author[Stone]{Branden Stone}
\address{Branden Stone, Department of Mathematics and Computer Science, Adelphi University, 1 South Avenue, Garden City, NY 11530-0701}
\email{bstone@adelphi.edu}

\begin{abstract} 
We introduce a recursive decomposition algorithm for the Betti diagram of a complete intersection using the diagram of a complete intersection defined by a subset of the original generators.  This alternative algorithm is the main tool that we use to investigate stability and compatibility of the Boij-S\"oderberg decompositions of related diagrams; indeed, when the biggest generating degree is sufficiently large, the alternative algorithm produces the Boij-S\"oderberg decomposition.  We also provide a detailed analysis of the Boij-S\"oderberg decomposition for Betti diagrams of codimension four complete intersections where the largest generating degree satisfies the size condition.
\end{abstract}


\maketitle

%------------------------------------------------------------------------
%------------------------------------------------------------------------
\section{Introduction}\label{sec:sec1}

Since its conception \cite{BS1,ES1}, \BS theory has blossomed into an active area of research in commutative algebra.
One part of the dual nature of this theory allows us to analyze numerical invariants of graded finite free resolutions.  Applications include the proof of the multiplicity conjecture \cite{HS,ES1}, a special case of Horrocks' conjecture \cite{ermanHorrocks}, and constraints on regularity \cite{McC}.  There are current efforts to extend \BS theory to Grassmannians \cite{samGrass} as well as expository notes on open questions and the state of the field \cite{ES,floystad-expos}. In the case of complete intersections, it was shown in \cite{reu} that the diagrams of complete intersections behave similarly to pure diagrams, creating a non-trivial sub-cone of the Boij-S\"oderberg cone. Further, in \cite{GJMRSW}, a complete structure theorem is given for complete intersections of codimension at most three. Recent work shows decompositions of homogeneous ideal powers stabilize in a meaningful way \cite{mayestang,whieldon}, and in \cite{NS}, the authors give some combinatorial interpretations of the decompositions produced by the Boij-S\"oderberg decomposition algorithm.

Despite all of this, not much is generally known about the structure of the decompositions for specific classes of modules; our understanding of the structure of Boij-S\"oderberg decompositions remains extremely limited.  In fact, even studying the Boij-S\"oderberg decomposition of a complete intersection defined by forms of degrees $\lst a c$ raises nuanced questions, and the decomposition depends on delicate relations among the $a_i$.  Our main result provides new insight into the way that these relations--in particular, the relationship between the largest degree to the smaller degrees--affect the decomposition.

Consider a homogeneous complete intersection ideal $I$ in a polynomial ring $S = \kk[\lst x n]$.  The goal of this paper is to find a relationship between the decompositions of diagrams $\bet(S/I)$ and $\bet(S/I\otimes \kk[y]/(g))$. In particular, Theorem~\ref{thm:recursive-stuff} indicates that the decomposition $\bet(S/I\otimes \kk[y]/(g))$ can be obtained from the decomposition of the diagram $\bet(S/I)$ under certain conditions, including a lower bound on the degree of $g$. The main tool we use to link the decompositions of these two diagrams is an alternative decomposition technique introduced in Algorithm~\ref{NewAlgorithm}. We discuss this new algorithm in Section~\ref{sec:newalg} while the rest of Section~\ref{sec:sec1} develops the necessary notation and tools. In particular, Subsection~\ref{intro:eliminorder} develops the concept of an ``elimination order'' as a prelude to the new algorithm. Section~\ref{sec:stable} shows that the Boij-S\"oderberg decomposition of $\bet(S/I\otimes \kk[y]/(g))$ stabilizes similarly to the ideal powers seen in \cite{mayestang}. A case study in codimension four is given in Section~\ref{sec:codim4}, extending the results of \cite{GJMRSW} and giving a partial structure theorem for the Boij-S\"oderberg decomposition of a complete intersection in codimension four.

%------------------------------------------------------------------------
%------------------------------------------------------------------------
\subsection{Notation}\label{intro:notation}

Over the polynomial ring $S = \kk[\lst x n]$, every finite-length $S$-module $M$ has a minimal graded free resolution of the form
\[
	\xymatrixrowsep{5mm}
	\xymatrixcolsep{5mm}
	\xymatrix
		{
			0  & M  \ar[l] & \ds \bigoplus_j S(-j)^{\beta_{0,j}} \ar[l] & \ds \bigoplus_j S(-j)^{\beta_{1,j}} \ar[l] & \cdots \ar[l] & \ds \bigoplus_j S(-j)^{\beta_{p,j}} \ar[l] & 0 \ar[l]
		}
\]
where the projective dimension $p$ is at most $n$ by Hilbert's Syzygy Theorem. The integer $\beta_{i,j}$ is the number of degree $j$ generators of a basis of the free module in the $i^{\text{th}}$ step of the resolution.  These $\beta_{i,j}$ are independent of the choice of minimal free resolution, and they are called the \defi{graded Betti numbers}.  The \defi{Betti diagram} of $M$ is defined to be 
\[
	\beta(M) = \begin{pmatrix}
		\vdots & \vdots & \vdots & \adots & \\
		\beta_{0,{-1}} & \beta_{1,0} & \beta_{2,1} & \cdots &\\
		\beta_{0,0} & \beta_{1,1} & \beta_{2,2} & \cdots & \\
		\beta_{0,1} & \beta_{1,2} & \beta_{2,3} & \cdots & \\
		\vdots & \vdots & \vdots &\ddots &
	\end{pmatrix}.
\]
Throughout we consider $\beta(M)$ as an element of the vector space \[V = \bigoplus_{i=0}^n \bigoplus_{j \in \Z} \Q.\] If $D \in V$, then we say that $D$ is a \defi{diagram}. 

In this paper, we are concerned with Betti diagrams of homogeneous complete intersection modules over the ring $S$ viewed as a standard graded $\kk$-algebra where $n \gg 0$.  Such a module is a quotient of $S$ by an ideal $I$ generated by a regular sequence $\lst f c$, and its free resolution is given by the Koszul complex $K_{\bullet}(\lst f c) = K_{\bullet}(f_1) \otimes \cdots \otimes K_{\bullet}(f_c)$. Because tensor products commute, we may assume that $I = (f_1,\ldots,f_c)$  is written in such a way that the degrees of the forms $f_i$ are nondecreasing from left to right. In particular, letting $a_i = \deg f_i$, the combinatorial construction of the Koszul complex makes it easy to verify that $\beta_{i,j}(S/I)$ is the number $i$-element subsets of $\{a_1,\ldots,a_c\}$ that sum to $j$. As such, we simplify notation by setting
\[
	\beta(a_1,\ldots,a_c) := \beta(S/I).  
\]

Observe that $\beta_{0,0}(a_1,\ldots,a_c) = \beta_{0,0}(S/I) = 1$. Moreover, the projective dimension and regularity can be calculated as, respectively, 
\[
	\pdim{S/I} = \codim(S/I) = c \quad \text{and} \quad \reg(S/I) = \sum_{k=1}^c a_k - c.
\]

%------------------------------------------------------------------------
%------------------------------------------------------------------------
\subsection{Boij-S\"oderberg Theory}\label{intro:BS}

Let $M$ be an $S$-module of finite length.	 We say $\dd \in \Z^{n+1}$ is a \defi{degree sequence} if $d_i < d_{i+1}$ for all $i$ and that $\dd \ls \dd'$ if $d_i \ls d_i'$ for all $i$.  A \defi{chain of degree sequences} is a totally ordered collection $\{\cdots < \dd^0 < \dd^1 < \cdots < \dd^s < \cdots \}$.  Viewing $\beta(M)$ as a diagram, we say that $\beta(M)$ is a \defi{pure diagram} if $\beta(M)$ has at most one non-zero entry in each column.  For example, if $S = \kk [x,y,z]$ and $M = \kk [x,y,z]/(x^2,y^2,z^2)$ then 
	\[
		\beta(M) = \begin{pmatrix}
			1 & - & - & - \\
			- & 3 & - & - \\
			- & - & 3 & - \\
			- & - & - & 1
		\end{pmatrix}
	\]
is a pure diagram. If $\beta(M)$ is a pure diagram, then for each nonnegative integer $i \leq \pdim M$ there exists an integer $d_i$ for which $\beta_{i,j}(M) \not= 0$ if and only if $j = d_i$.  In this case we say that $\beta(M)$ is a \defi{pure diagram of type $\dd = (d_0,d_1, \dots, d_n)$}.  Further, if $\dd$ is a degree sequence and $M$ is any finite length module with a pure diagram of type $\dd$, then entries of $\beta(M)$ are an integer multiple of the diagram $\pi(\dd)$ given by 
\[ 
\pure \dd_{ij} = 
		\begin{cases}
				\prod_{k \not = i}\frac1{|d_i - d_k|}, & j = d_i \\
					 0 ,									& j \not = d_i
					 \end{cases}
	\]
(see \cite{HK}), and in \cite{ES1}, the authors show such a module exists for each degree sequence $\dd$.

For example, if $\dd = (0,2,4,6)$, then {}
\[
	\pure \dd = \begin{pmatrix}
		1/48 & - & - & - \\
		- & 1/16 & - & - \\
		- & - & 1/16 & - \\
		- & - & - & 1/48
	\end{pmatrix}.
\]

In proving the conjectures of M.\ Boij and J.\ S\"oderberg \cite{BS2},  D.\ Eisenbud and F.O.\ Schreyer show there is a unique decomposition of Betti tables in terms of $\pure \dd $ \cite{ES1}. 

\begin{thm}[\cite{BS1,ES1}]\label{thm:uniqueDecomp}
	Let $S = \kk [\lst x n]$ and $M$ an $S$-module of finite length.  Then there is a unique chain of degree sequences $\{\dd^1 < \cdots < \dd^s\}$ and a unique set of scalars $z_i \in \Q_{>0}$ such that 
	\[
		\beta(M) = \sum_{i = 1}^s z_i \pure{\dd^i}.
	\]
\end{thm}

The unique decomposition in Theorem \ref{thm:uniqueDecomp} respects the partial order (see \cite[Definition 2]{BS2}) of the degree sequences $\dd^i$ and is obtained by applying the greedy algorithm to a special chain of degree sequences.  We formalize this algorithm as follows.

\begin{algorithm}[Totally Ordered Decomposition Algorithm \cite{ES1}] \label{StandardDecomp}
Let  $S$ be $\kk [\lst x n]$ and $M$ be a finitely generated $S$-module of finite length.  Set $\beta = \beta(M)$.
	\begin{enumerate}
		\item[Step 1:] Identify the minimal degree sequence $\ul \dd$ of $\beta$;
		\item[Step 2:] Choose $q > 0 \in \Q$ maximal with respect to the condition that $\beta - q\pure{\ul \dd}$ has non-negative entries;
		\item[Step 3:] Set $\beta = \beta - q\pure{\ul \dd}$;
		\item[Step 4:] Repeat Steps 1-3 until $\beta$ is a pure diagram;
		\item[Step 5:] Write $\beta(M)$ as a sum of the the $q\pure{\ul \dd}$ obtained in the above steps.
	\end{enumerate}		
\end{algorithm}

\begin{example}\label{ex:EarlyCodim3DecompPart1} Consider $D = \beta(2,3,4)$.  The algorithm above produces the output
\[D = 42 \pure{0,2,5,9} + 12 \pure{0,3,5,9} + 36 \pure{0,3,6,9} + 12 \pure{0,4,6,9} + 42 \pure{0,4,7,9},\]
with degree sequences chosen by the algorithm in order from left to right.
\end{example}

We note that our choice of $\pure{\dd}$ differs from the choices used in \cite{BS1} and \cite{ES1}. In \cite{BS1}, they choose the pure diagram with $\beta_{0,0}=1$; in \cite{ES1}, they choose the smallest possible pure diagram with integral entries. Since the pure diagrams with degree sequence $\dd$ form a one-dimensional vector space, this different choice only affects the coefficients that arise in the algorithm. 

Let $D \in V$ be a diagram. Define the \defi{dual} of $D$, denoted $D^*$, via the formula 
	\[
		(D^*)_{ij} = D_{n-i,-j},
	\] 
and define \defi{twist} $D(r)$ via the formula 
	\[
		D(r)_{ij} = D_{i,r+j}.
	\]  
These definitions mimic the functors $\Hom_S(-,S) = -^*$ and $-\otimes S(r)$ for modules; one may check that $\beta(M^*) = \beta(M)^*$ and $\beta(M(r)) = (\beta(M))(r)$.  In particular, if $M$ is a Gorenstein module of finite length, the Betti diagram will be self-dual up to shift by $\Reg(M)$.

\begin{thm}[Symmetric Decomposition \cite{EKS},\cite{NS}]\label{symmetric} 
	Let $S=\kk [x_1,\ldots,x_n]$ and $M$ be a Gorenstein $S$-module of finite length.  Then the decomposition of $\beta(M)$ via Algorithm \ref{StandardDecomp} is symmetric; i.e., 
		\[
			\beta(M) = a_1 \pure{\dd^1} + a_2 \pure{\dd^2} + \cdots + a_2 \pure{\dd^2}^*(r) + a_1 \pure{\dd^1}^*(r)
		\] 
 	where $r = \Reg(M)$.
\end{thm}

%------------------------------------------------------------------------
%------------------------------------------------------------------------
\subsection{Elimination Order}\label{intro:eliminorder}

The concept of an elimination order was first introduced in \cite{GJMRSW}. We give the basics here.

\begin{definition} 
Given a diagram $D \in V$, its \defi{elimination table} has as its $(i,j)^{\text{th}}$ entry the integer $k$ such that the $k^{\text{th}}$ iteration of Algorithm~\ref{StandardDecomp} applied to $D$ is the first iteration such that the $(i,j)^{\text{th}}$ entry of $D$ becomes zero.  The \defi{elimination order of $D$} is the sequence $\elimorder{D}$ with $k^{\text{th}}$ entry \[
\{(i_{k_1},j_{k_1}),\ldots,(i_{k_\ell},j_{k_\ell})\}
\]
where the $k^{\text{th}}$ step of Algorithm~\ref{StandardDecomp} eliminates $D_{i_{k_1},j_{k_1}},\ldots,D_{i_{k_\ell},j_{k_\ell}}$.
We denote by $\elimsize{D}$ the number of pure diagrams in the chain used by Algorithm~\ref{StandardDecomp}; $\elimsize{D}$ is therefore the maximal integer appearing in the elimination table of $D$.
\end{definition}

The elimination table is a means of recording the elimination order of the row and column position according to Algorithm~\ref{StandardDecomp}.  Given any diagram $D \in V$, the sequence of pure diagrams appearing in the Boij-S\"oderberg decomposition of $D$ can be obtained recursively from the elimination table. Indeed, in Algorithm~\ref{StandardDecomp}, the degree sequence of the pure diagram corresponding to the $t^{\text{th}}$ iteration is given by the sequence of least degrees in each column after $t-1$ eliminations. We may read this in the elimination table as entries of least degree in each column (i.e., highest up on the page) among those with value at least $t$.

\begin{example}\label{ex:EarlyCodim3DecompPart2} Let $D = \beta(2,3,4)$. Algorithm~\ref{StandardDecomp} gives the elimination table 
\[
\begin{pmatrix} 5 & \text{.} & \text{.} & \text{.} \\
				\text{.} & 1 & \text{.} & \text{.} \\
				\text{.} & 3 & \text{.} & \text{.} \\
				\text{.} & 5 & 2 & \text{.} \\
				\text{.} & \text{.} & 4 & \text{.} \\
				\text{.} & \text{.} & 5 & \text{.} \\
				\text{.} & \text{.} & \text{.} & 5 
\end{pmatrix}
\]
and elimination order 
\[
	\elimorder{D} = \big(\{(1,2)\},\{(3,5)\},\{(2,3)\},\{(4,6)\}, \{(0,0),(3,4),(5,7),(6,9)\}\big).
\]
\end{example}

Observe that the only entry of $\elimorder{D}$ in Example~\ref{ex:EarlyCodim3DecompPart2} that isn't a singleton set is the fifth and final entry; this corresponds to the last step of Algorithm~\ref{StandardDecomp}, where the bottom line of $D$ is removed with the final pure diagram.  This behavior is quite nice, but not guaranteed.  Indeed, one may check that $D = \beta(2,3,5,7)$ has elimination order
\begin{align*}
\elimorder{D} = \big( &
\{(1,2)\},
\{(2,5)\},
\{(2,7)\},
\{(1,3),(3,10)\},
\{(2,8)\},
\{(1,5),(3,12)\},\\
& \quad
\{(2,9)\},
\{(2,10)\},
\{(3,14)\},
\{(0,0),(1,7),(2,12),(3,15),(4,17)\}
\big),
\end{align*}
so that $\#(\elimorder{D}_4) = \#(\elimorder{D}_6) = 2$, but $4 < 6 < \elimsize{D} = 10$.  We give this behavior a name below.

\begin{definition} Given a diagram $D \in V$, we say that \defi{mass elimination} occurs if $\#\left(\elimorder{D}_k\right) > 1$ for some $k \lneq \elimsize{D}$. 
\end{definition}

Note that mass elimination does not occur for complete intersections of codimension at most three. At each stage of Algorithm~\ref{StandardDecomp}, the choice of $q$ will only eliminate one entry. This is not the case for codimension four and above.

%------------------------------------------------------------------------
%------------------------------------------------------------------------
\section{Decomposing Complete Intersections}\label{sec:newalg}

For the rest of the paper, we assume $a_1 \leq a_2 \leq \cdots \leq a_c \leq a_{c+1}$ and we set $a = a_1 + \cdots + a_c$.

In this section, we create a decomposition algorithm for $\beta(a_1,\ldots, a_c,a_{c+1})$ that uses the decomposition of $\beta(a_1,\ldots,a_c)$ from Algorithm~\ref{StandardDecomp} and $a_{c+1}$ as initial inputs. We then describe when the new algorithm produces the traditional Boij-S\"oderberg decomposition of $\beta(a_1,\ldots,a_{c+1})$ that arises via Algorithm~\ref{StandardDecomp}.

\begin{remark}[Notation]\label{rmk:notation} 
Given a degree sequence $\dd = (d_0, d_1, \ldots, d_c)$, we define $\concat(\dd,N) = (d_0, \ldots, d_c, N)$; when $N > d_c$, this is a degree sequence in codimension $c+1$.  When it is understood that we are working with a complete intersection of generating degrees $a_1 \leq a_2 \leq \cdots \leq a_c \leq a_{c+1}$, we use the notation $\beta^{\ell} = \beta(a_1,\ldots,a_{\ell})$, where $\ell \leq c+1$, to simplify notation in prose and inductive arguments. In order to designate specific entries of the input diagram $\beta^c$ and its decomposition via Algorithm~\ref{StandardDecomp}, we set up the following notational conventions.  Suppose $\beta^c$ has no instances of mass elimination.   Let $\elims = \elimsize{\beta^c}$.  Then $\elimorder{\beta^c}_s = \{(i_s,j_s)\}$ for $1 \leq s \leq \elims -1$. We set $b_s = \left({\beta^c}\right)_{i_s,j_s}$ so that it is the entry of $\beta^c$ that is eliminated in step $s$.  For $s = \elims$, we define $b_\elims = {\beta^c}_{c,a}$ so that $i_\elims = c$ and $j_\elims = a$.  Finally, for $1 \leq s \leq \elims$, we define $p_s = \pure{\dd^s}_{i_s,j_s}$, the relevant entry of the pure diagram used to eliminate $b_s$.  It is useful to notice that ${\dd^s}_{i_s} = j_s$, and when $a_{c+1} \geq a$, $\left({\beta^{c+1}}\right)_{i_s,j_s} = b_s$.
 \end{remark}

%------------------------------------------------------------------------
%------------------------------------------------------------------------
\subsection{New Algorithm}

Our goal is to create an algorithm that decomposes $\beta^{c+1}$ that respects the elimination order of $\beta^c$.  For ease of discussion, we divide it into three phases. \ref{p1} consists of three steps, where the algorithm follows the previous elimination order to eliminate the top left entries.  That is, the entries $b_k$ from $\beta(\lst a c)$ appear as entries in the new diagram $\beta(\lst a {c+1})$, and we eliminate them in the same order (taking care to choose a specific entry to eliminate at step $\elims$). \ref{p2} consists of the next $c-1$ steps, where the degree sequences are chosen to eliminate specific entries of the diagram from right to left.  \ref{p3} consists of the remainder of the algorithm and uses degree sequences symmetric to those in \ref{p1} to eliminate as many of the remaining elements as possible. A priori, it is possible that this algorithm terminates with an error diagram $E$, by which we mean that \ref{p3} is not assumed to finish eliminating the bottom of the diagram. Thus, we say that Algorithm~\ref{NewAlgorithm} \defi{decomposes $\beta(\lst a {c+1})$} provided $E$ is the zero diagram. 

\begin{algorithm}[New Algorithm]\label{NewAlgorithm}
	Let $c>1$. Consider a complete intersection 
	\[
		R = \kk[\lst x c, x_{c+1}]/(f_1,\ldots,f_c,f_{c+1}),
	\]
	where $\deg(f_i) = a_i$ and $a_1 \ls a_2 \ls \cdots \ls a_c \ls a_{c+1}$ with $\A = a_1 + a_2 + \cdots + a_c$. A decomposition of $\bet^{c+1} = \betti{\lst a c, a_{c+1}}$ 
	is given as follows.
	\begin{description}
		\item [Phase~1\label{p1}] Elimination of $\beta^{c+1}$ with respect to $\elimorder{\beta^c}$. 
		\begin{enumerate}[label=\textit{Step 1.\arabic*:},ref=1.\arabic*]
			\item \label{p1.1}
				Consider the decomposition of $\beta^c$ according to Algorithm~\ref{StandardDecomp} and let $\elims = \elimsize{\betti{\lst a c}}$ In particular, let 
				\[
					\bet^c = \sum_{s = 1}^\elims z_s \pure{\dd^s}.
				\]

			\item \label{p1.2}
				Make new degree sequences $\ee^s = \concat(\dd^s,a+a_{c+1})$ for $1 \ls s \ls \elims$. 

			\item \label{p1.3}
				Use each new $\ee^s$ degree sequence to eliminate the entries of $\bet^{c+1}$ according to the elimination order $\elimorder{\bet^c}$. That is, in step $s$, eliminate ${\beta^{c+1}}_{i_s,j_s} = b_s$.
				Set the respective coefficients in $\Q$ equal to $y_\sigma$. In particular, let
				\[
					P_1 = \sum_{s = 1}^{\elims} y_s \pure{\ee^s}. 
				\]
				Then we have that $\bet^{c+1} = P_1 + E$ for some error diagram $E$.
		\end{enumerate}
		\item [Phase~2\label{p2}] Eliminate from right to left along the columns.
		\begin{enumerate}[label=\textit{Step 2.\arabic*:},ref=2.\arabic*]
			\item\label{p2.1} If $a_c = a_{c+1}$, proceed to \ref{p3}. Otherwise define the degree sequences as
				\[ \ee_i^{\elims + k} =
					\begin{cases}
						 \ee_i^{\elims + k -1}, 		\quad &\text{for } i \not= c-k+1 \\
						 \sum_{l=1}^{i-1} a_l+a_{c+1}, 	\quad &\text{for } i = c-k+1
					\end{cases}
				\]
				for $1 \ls k \ls c-1$. 

			\item\label{p2.2} Use the degree sequences in Step \ref{p2.1} and choose each $y_s$ to eliminate the top entry in the $c-k^{\text{th}}$ column, starting with $k = 1$. In particular, let 
			\[
				P_2 = \sum_{s = \elims+1}^{\elims+c-1} y_s \pure{\ee^s}, 
			\]
			where $\ee^\elims$ is defined in Step~\ref{p1.2}. Then we have $\bet^{c+1} = P_1 + P_2 + E$ for some error diagram $E$.
		\end{enumerate}
		\item [Phase~3\label{p3}] Eliminate using the dual of the degree sequences in \ref{p1}.
		\begin{enumerate}[label=\textit{Step 3.\arabic*:},ref=3.\arabic*]
			\item\label{p3.1} Complete the elimination according to Theorem~\ref{symmetric}. Letting
				\[
					P_3 = \sum_{s = \elims+c}^{2\elims + c -1} y_{s} \pure{\ee^{s}},
				\]
				where $\pure{\ee^{s}} = \pure{\ee^{2\elims+c-1-s}}^*(a+a_{c+1}-c-1)$ from \ref{p1}. Here, each $y_s$ is chosen to eliminate the position given by the elimination order of $\bet^c$.

			\item\label{p3.2} Combining all three Phases, we have 
			\[
				\bet^{c+1} = P_1 + P_2 +P_3 + E
			\]
			for some error diagram $E$. 
		\end{enumerate}
	\end{description}
\end{algorithm}

For a general understanding of Algorithm~\ref{NewAlgorithm}, consider the following example.   The full details are given in Example~\ref{ex:stabilze}.

\begin{example}\label{ex:EarlyCodim3DecompPart3} 
Let $D = \bet(2,3,4,a_4)$ where $a_4 \geq 4$. From Example~\ref{ex:EarlyCodim3DecompPart2} we know the elimination order of $\bet(2,3,4)$. With this information, we can run Algorithm~\ref{NewAlgorithm} on the diagram $D$; the table in Figure~\ref{fig:NewAlgPic} is the elimination table of this process. 

\begin{figure}[H]
\begin{tiny}\setlength{\arraycolsep}{1pt}
	\[
	\begin{pmatrix}
	     \ensquare{12}&\text{.}&\text{.}&\text{.}&\text{.}\\
	     \text{.}&\encircle 1&\text{.}&\text{.}&\text{.}\\
	     \text{.}&\encircle 3&\text{.}&\text{.}&\text{.}\\
	     \text{.}&\enwithdiamond 7&\encircle 2&\text{.}&\text{.}\\
	     \text{.}&\text{.}&\encircle 4&\text{.}&\text{.}\\
	     \text{.}&\text{.}&\enwithdiamond 6&\text{.}&\text{.}\\
	     \text{.}&\text{.}&\text{.}&\encircle 5&\text{.}\\
	     \text{.}&\text{.}&\text{.}&\text{.}&\text{.}\\
	     		&\vdots&\vdots&\vdots&\\
	     \text{.}&\text{.}&\text{.}&\text{.}&\text{.}\\
	     \text{.}&\ensquare{12}&\text{.}&\text{.}&\text{.}\\
	     \text{.}&\text{.}&\ensquare{8}&\text{.}&\text{.}\\
	     \text{.}&\text{.}&\ensquare{10}&\text{.}&\text{.}\\
	     \text{.}&\text{.}&\ensquare{12}&\ensquare{9}&\text{.}\\
	     \text{.}&\text{.}&\text{.}&\ensquare{11}&\text{.}\\
	     \text{.}&\text{.}&\text{.}&\ensquare{12}&\text{.}\\
	     \text{.}&\text{.}&\text{.}&\text{.}&\ensquare{12}\end{pmatrix}
	\]
  \end{tiny}
	\caption[A pictorial representation of Algorithm~\ref{NewAlgorithm}.]{A pictorial representation of Algorithm~\ref{NewAlgorithm}. \ref{p1} entries are marked with {\tiny$\encircle{\phantom 3}$}. Similarly \ref{p2} and \ref{p3} are marked with {\tiny$\enwithdiamond{\phantom 3}$} and {\tiny$\ensquare{\phantom{10}}$} respectively.}
	\label{fig:NewAlgPic}
\end{figure}

Given the diagram $D$, in \ref{p1} we calculate the decomposition and the elimination order of $\bet(2,3,4)$. 
Using this information, we form new degree sequences and choose coefficients to target the entries that line up with
the elimination order of $\bet(2,3,4)$. 
In \ref{p2}, we target the rest of the ``old'' diagram $\bet(2,3,4)$ working from right to left. 
In \ref{p3}, we form degree sequences by calculating the dual of the pure diagrams from \ref{p1} and choose positions to 
eliminate by reversing the elimination sequence from \ref{p1}. 
In fact, Algorithm~\ref{NewAlgorithm} decomposes $\beta(2,3,4,a_4)$ as long as $a_4 \geq 4$.  When $a_4 > 12$, the degree sequences form a chain and the coefficients are all positive, so the algorithm produces the traditional Boij-S\"oderberg decomposition of the diagram.  Thus $\elimorder{\beta(2,3,4,13)}$ is compatible with $\elimorder{\beta(2,3,4)}$.  However, when $4 \leq a_4 \leq 12$, the elimination order is different and some of the coefficients produced by Algorithm~\ref{NewAlgorithm} will be negative or, in the case of $a_4 = 12$, zero.
\end{example}

Despite the fact that Algorithm~\ref{NewAlgorithm} can produce negative coefficients, it will always decompose the diagram. 

 \begin{thm}\label{thm:Ezero} 
 	Algorithm~\ref{NewAlgorithm} decomposes $\bet(\lst a {c+1})$. In particular, $E$ is always the zero diagram. 
 \end{thm} 
\begin{proof}
	At each stage of Algorithm~\ref{NewAlgorithm} a scalar multiple of a pure diagram is subtracted from the previous diagram, eliminating at least one entry.  After the penultimate step, the non-zero entries of the resulting diagram $D$ corresponds to the degree sequence $\ee^{2\elims+c-1}$. It is enough to show that $s\cdot \pure{\ee^{2\elims+c-1}} = D$ for some scalar $s$. To do this, notice each stage of Algorithm~\ref{NewAlgorithm} produces a diagram satisfying the Herzog-Kuhl equations \cite{BS1} where our codimension is $c+1$,
	\begin{align*}
		\sum_{i,j}(-1)^i j^0 \beta_{i,j} &= 0\\
		\sum_{i,j}(-1)^i j^1 \beta_{i,j} &= 0\\
		&\,\,\,\,\,\vdots \\
		\sum_{i,j}(-1)^i j^c \beta_{i,j} &= 0.
	\end{align*}
	Since each column of $D$ has at most one non-zero entry we can write the above system as the following matrix equation,
	\begin{align*}
		\begin{bmatrix}
			1 & 1 & \cdots & 1 \\
			e_0 & e_1 & \cdots & e_{c+1} \\
			e_0^2 & e_1^2 & \cdots & e_{c+1}^2 \\
			\vdots & \vdots &  & \vdots \\
			e_0^c & e_1^c & \cdots & e_{c+1}^c \\
		\end{bmatrix}%
		\begin{bmatrix}
			d_0 \\
			-d_1 \\
			d_2 \\
			\vdots \\
			\pm d_{c+1}
		\end{bmatrix}%
		=
		\begin{bmatrix}
			0 \\
			0 \\
			0 \\
			\vdots \\
			0			
		\end{bmatrix},				
	\end{align*}
	where $\ee^{2\elims+c-1} = (e_0, \lst e {c+1})$ and $d_i$ are the entries of $D$ in column $i$, degree $e_i$. The last step of the algorithm chooses an $s$ such that $D' = D-s\pure{\ee^{2\elims+c-1}}$ has an entry greedily eliminated from $D$, say $d_i$. As the shape of $D'$ still corresponds to $\ee^{2\elims+c-1}$, the entries $d'_0, \lst{d'}{c+1}$ satisfy the matrix problem above. Since $d'_i = 0$, we can ignore the $i$th column of the matrix, obtaining an invertible Vandermonde matrix. This new homogeneous system has only the trivial solution, forcing $D'=0$.
\end{proof}

The following results are useful in determining when the decomposition lines up with the \BS decomposition, as described in Corollary~\ref{cor:algisBS}. 

\begin{dfn}[\cite{EKS}]\label{dfn:check} For a degree sequence $\ee = (\ee_0,\ee_1,\ldots,\ee_{c+1})$ with $\ee_0 = 0$, set
\[ \ee^\vee = (\ee_{c+1} - \ee_{c+1}, \ee_{c+1} - \ee_{c}, \ldots, \underbrace{\ee_{c+1} - \ee_{c+1 - k}}_{\text{position $c+1+k$}}, \ldots, \ee_{c+1} - \ee_0
).
\]
\end{dfn}

From \cite[Proposition 2.4]{EKS}, it follows that 
\begin{equation}\label{eq:check-star}
\pure{\ee}^*(a + a_{c+1} - c) = \pure{\ee^\vee}.
\end{equation}	

\begin{lem}\label{lem:symmetry} The degree sequences and coefficients from Algorithm~\ref{NewAlgorithm} are symmetric.  That is, $\ee^s = (\ee^{N-s + 1})^\vee$ and $y_s = y_{N-s + 1}$ where $N = 2\elims + c - 1$.
\end{lem}
\begin{proof}
	Assume that $\elims + k \ls N - \elims -k$, this forces $k+1 < c-k+1$. Starting the index of the degree sequence at 0, we have 
	\begin{align*}
		\ee^{\elims+k} = \Bigg(0, a_c, a_c + a_{c-1}, &\ldots,
		\underbrace{\sum_{\ell = k+1}^{c} a_\ell}_{\text{index } c-k},\\ 
		& \underbrace{a_{c+1} + a - \sum_{\ell = c-k+1}^c a_\ell }_{\text{index } c-k+1}, 
		\ldots, a_{c+1} + a - \sum_{\ell = c}^c a_\ell , a_{c+1} + a\Bigg).
	\end{align*}
	After applying the dual we should have the following degree sequence, 
	\begin{align*}
		\ee^{\elims+(c-k)} = \Bigg(0, a_c, a_c + a_{c-1}, &\ldots, 
		\underbrace{\sum_{\ell = c-k+1}^{c} a_\ell}_{\text{index } k},\\
		&\underbrace{a_{c+1} + a - \sum_{\ell = k+1}^c a_\ell }_{\text{index } k+1}, 
		\ldots, a_{c+1} + a - \sum_{\ell = c}^c a_\ell , a_{c+1} + a \Bigg).
	\end{align*}
	Calculating the dual, we find that $(\ee^{\elims+k})^\vee = \ee^{\elims+(c-k)}$. 

	Next, we show $y_s = y_{N-s+1}$.

	Since $\beta^{c+1} = \left(\beta^{c+1}\right)^*(a + a_{c+1} - c)$, we have that $\left(\beta^{c+1}\right)^*(a + a_{c+1} - c)= \sum y_s \pure{\ee^s}$.  However, $\left(\beta^{c+1}\right)^*(a + a_{c+1} - c) = \sum y_s \pure{\ee^s}^*(a + a_{c+1} - c)$ as well.  Since $\pure{\ee^s}^*(a + a_{c+1} - c) = \pure{\ee^{N - s+1}}$ and the set $\{\pure{\ee^s}\}$ is a basis for $V$, it follows that $y_s = y_{N-s+1}$.
\end{proof}

\begin{prop}\label{lem:chain} If ${a_{c+1}} \geq a$, then the set of degree sequences $\{\ee^s\}$ is totally ordered. 
\end{prop} 

\begin{proof} 
	The degree sequences $\ee^1,\ldots,\ee^\elims$ in \ref{p1} form a chain because the degree sequences $\dd^1,\ldots,\dd^\elims$ were obtained from Algorithm~\ref{StandardDecomp} and thus form a chain themselves.

	The degree sequences $\ee^{\elims + 1}$ and $\ee^{\elims}$ may differ only in position $c$.  There, we have $\ee^{\elims + 1}_c = \sum_{i = 1}^{c-1} a_i + a_{c+1} \geq \sum_{i = 1}^c a_i$.  If $a_{c+1} = a_c$, the algorithm skips \ref{p2} and there is nothing to show.  

	Otherwise, $a_{c+1} \geq a > a_c$ implies that $\ee^{\elims + 1} > \ee^\elims$.
	Further, the degree sequences $\ee^{\elims + k}$ and $\ee^{\elims + k + 1}$ differ only in position $c-k$.  There, we have 
	\begin{align*}
	\ee^{\elims + k+1}_{c-k} &= \sum_{i = 1}^k a_i + a_{c+1} \\
	&\geq \sum_{i = c-k}^{c} a_i = \ee^{\elims + k}_{c-k}
	\end{align*} since $a_{c+1} \geq a > \sum_{i = c-k}^{c} a_i$.

	Therefore, the degree sequences from \ref{p1} and \ref{p2} are totally ordered.  By Lemma~\ref{lem:symmetry}, it follows that the degree sequences from all three phases form a totally ordered set.
\end{proof}

\begin{cor} If $a_{c+1} \geq a$ and Algorithm~\ref{NewAlgorithm} produces a decomposition of $\beta(a_1,\ldots,a_{c+1})$ where each $y_s$ is non-negative, then that decomposition agrees with the decomposition of $\beta(a_1,\ldots,a_{c+1})$ obtained by Algorithm~\ref{StandardDecomp}.
\end{cor}

We next formalize a relationship between coefficients in the Algorithm~\ref{StandardDecomp} decompositions of $\beta(\lst a c)$ and $\beta(\lst a {c+1})$, using \ref{p1} and \ref{p3} of Algorithm~\ref{NewAlgorithm}.

\begin{definition} 
	Consider a diagram $D = \beta(a_1,\ldots,a_c)$ with decomposition
	\[
		\beta(a_1,\ldots,a_c) = \sum_{s = 1}^\elims z_s \pi(\dd^s),
	\]
	obtained from Algorithm~\ref{StandardDecomp}. We define the \defi{remainders of $D$ relative to $a_{c+1}$} 
	to be the numbers $r_s$ such that $y_s = a_{c+1} z_s- r_s$, where the $y_s$ are the coefficients obtained by applying Algorithm~\ref{NewAlgorithm} to $\beta(a_1,\ldots,a_{c+1})$.
\end{definition}

\begin{theorem}\label{thm:recursive-stuff}
	Consider $\beta(a_1,\ldots,a_c)$, the Betti diagram of the complete intersection ideal generated in degrees $a_1\leq \ldots \leq a_c$ of elimination size $\elims$, for which the decomposition obtained from Algorithm~\ref{StandardDecomp},\[
	\beta(a_1,\ldots,a_c) = \sum_{s = 1}^\elims z_s \pi(\dd^s),
	\]
	has no instances of mass elimination.

	Let $\beta^{c+1} = \beta(a_1,\ldots,a_c,a_{c+1})$ for $a_{c+1} \geq a_c$, and set $a = \sum_{i = 1}^{c} a_i$.  Given the decomposition obtained from Algorithm~\ref{NewAlgorithm} and the remainders $r_s$ of $\beta(a_1,\ldots,a_c)$ relative to $a_{c+1}$, then
	\[
		\beta^{c+1} = \sum_{s = 1}^N y_s \pi(\ee^s),
	\] 
	and for $1 \leq s \leq \elims$, $y_s = z_s a_{c+1} - r_s$ with $r_s$ defined recursively in terms of the previous remainders.  Indeed, $r_1 = (j_1 -a)z_1$ and
	\[
		r_k = \left(\frac{j_k-a}{p_k}b_k-\sum_{s = 1}^{k-1}\frac{\pi(\dd^s)_{i_k,j_k}}{p_k}r_s\right)
	\] 
	for $2 \leq k \leq \elims$, where $b_k, i_k, j_k$ and $p_k$ are described in Remark~\ref{rmk:notation}.

	Furthermore, $a_{c+1} > \max\{a,\frac{r_1}{z_1},\ldots,\frac{r_\elims}{z_\elims}\}$ implies $y_s > 0$.

\end{theorem}

First we collect a few useful observations.

\begin{lemma}\label{lem:easy-obs} Given the hypotheses and notation of Theorem~\ref{thm:recursive-stuff}, 
\begin{enumerate}[label=(\alph*)]
\item\label{easy-obs:a} there is an inequality $a - j_k > 0$, and
\item\label{easy-obs:b} $\pure{\ee^s}_{i_k,j_k} = 0$ or $\pure{\ee^s}_{i_k,j_k} = \pure{\dd^s}_{i_k,j_k}/(a+a_{c+1} - j_k)$.
\end{enumerate}
\end{lemma}

\begin{proof} For \eqref{easy-obs:a}, notice that if $j_k$ is in column $i_k$, then $j_k$ is a sum of $i_k$ elements of $\{a_1,\ldots,a_c\}$, so $j_k < a$.  For \eqref{easy-obs:b}, observe that $\pi(\ee^s)_{i_k,j_k} = 0$ if $\ee^s_{i_k} \not = j_k$, while if $\ee^s_{i_k} = j_k$ (which means that $\dd^s_{i_k} = j_k$ as well), \[\pure{\ee^s}_{i_k,j_k} = \frac1{\prod_{\ell \not = i_k}|\ee^s_{\ell} -\ee^s_{i_k}|} = \pure{\dd^s}_{i_k,j_k}/(a+a_{c+1} - j_k).\]
\end{proof}

\begin{proof}[Proof of Theorem~\ref{thm:recursive-stuff}] 
Observe that Algorithm~\ref{NewAlgorithm} gives $b_1 = z_1 p_1$ and $b_1 = y_1 p_1/(a+a_{c+1} -j_1)$.  Thus $y_1 = (a+a_{c+1} - j_1)z_1 = a_{c+1}z_1 - (j_1-a)z_1$; that is, $r_1 = (j_1 -a)z_1$.

Fix $1 < k < \elims$ and suppose that $y_{k-1} = z_{k-1} a_{c+1} - r_{k-1}$. 

Using Algorithm~\ref{NewAlgorithm}, we have that $b_k = \sum_{s = 1}^k y_s \pure{\ee^s}_{i_k,j_k}$. Hence 
\begin{equation}\label{eq:substitution} y_k q_k = b_k - \sum_{s = 1}^{k-1} y_s \pure{\ee^s}_{i_k,j_k}.\end{equation} Furthermore, $q_k = \frac{p_k}{a+a_{c+1}-j_k}$. From these equations and Lemma~\ref{lem:easy-obs}, we obtain
\begin{align*}
y_k &= \frac{a + a_{c+1} - j_k}{p_k} \left(b_k - \sum_{s = 1}^{k-1} y_s \pure{\ee^s}_{i_k,j_k}\right)\\
&= \frac{a + a_{c+1} - j_k}{p_k}b_k - \frac{a + a_{c+1} - j_k}{p_k}\sum_{s = 1}^{k-1} y_s \pure{\ee^s}_{i_k,j_k}\\
&= \frac{a + a_{c+1} - j_k}{p_k}b_k - \sum_{s = 1}^{k-1}\frac{\pure{\dd^s}_{i_k,j_k}}{p_k} y_s.
\end{align*}

Now we apply the induction hypothesis:
\begin{align*}
y_k &= \frac{a + a_{c+1} - j_k}{p_k}b_k - \sum_{s = 1}^{k-1}\frac{\pure{\dd^s}_{i_k,j_k}}{p_k} (z_s a_{c+1}-r_s)\\
&= \left(\frac{b_k}{p_k} - \sum_{s = 1}^{k-1}\frac{\pure{\dd^s}_{i_k,j_k}}{p_k}z_s\right)a_{c+1} + \left(\frac{a-j_k}{p_k}b_k + \sum_{s = 1}^{k-1}\frac{\pure{\dd^s}_{i_k,j_k}}{p_k}r_s\right)\\
&= z_k a_{c+1} - \left(\frac{j_k-a}{p_k}b_k - \sum_{s = 1}^{k-1}\frac{\pure{\dd^s}_{i_k,j_k}}{p_k}r_s\right).
\end{align*}

Hence $r_k = \left(\frac{j_k-a}{p_k}b_k-\sum_{s = 1}^{k-1}\frac{\pure{\dd^s}_{i_k,j_k}}{p_k}r_s\right)$ is recursively defined, and if $a_{c+1} > \frac{r_i}{z_i}$, we have that $y_i > 0$.
\end{proof}

\begin{cor}\label{cor:algisBS}
	When $a_{c+1} > \max\left\{a,\frac{r_1}{z_1},\ldots,\frac{r_\elims}{z_\elims}\right\}$, then Algorithm~\ref{NewAlgorithm} produces the Boij-S\"oderberg decomposition as in Algorithm~\ref{StandardDecomp}.
\end{cor}
\begin{proof}  By Theorem~\ref{thm:Ezero} and Proposition~\ref{lem:chain}, we have a decomposition with degree sequences that form a chain.  It is enough to show that the coefficients in the decomposition are positive.  

By choice of $a_{c+1}$, the coefficients $y_s$ for $1 \leq s \leq \elims$ are positive and eliminate the top part of the diagram. If this elimination order differs from the Boij-S\"oderberg decomposition, then at some step $s$, $y_s$ was not the greedy choice.  However, $\ee^s$ is the necessary degree sequence because it is the topmost degree sequence at this step.  Together, these statements mean that, by Algorithm~\ref{StandardDecomp}, there exists a maximal $q > y_s \in \Q$ such that when subtracting $q\pure{{\ee^s}}$ the resulting diagram has non-negative entries. But since subtracting $y_s\pure{\ee^s}$ eliminates entry $b_s$, then any $q > y_s$ results in a diagram with a negative entry in this position, a contradiction. 

As such, up to $\elims$, the algorithm respects the Boij-S\"oderberg decomposition. Further, our choice of $a_{c+1}$ forces the \ref{p2} elimination order to also respect the Boij-S\"oderberg decomposition, as any other order would create non-pure degree sequences. This forces all of the coefficients to be positive rational values and hence must be the Boij-S\"oderberg decomposition.
\end{proof}

By Theorem \ref{thm:recursive-stuff}, for sufficiently large $a_{c+1}$, we have characterized the behavior of the degree sequences in all phases of the algorithm, as well as the coefficients in \ref{p1} and \ref{p3}. In particular, we know both the degree sequences and the known coefficients are completely determined by the decomposition of $\beta(a_1,\ldots,a_c)$. Unfortunately, the coefficients in \ref{p2} are still elusive. Using \cite{M2} we were able to generate enough examples to form the following conjecture. We also show this conjecture holds for codimension at most 3 in Corollary~\ref{cor:conj}.
 
\begin{conjecture}\label{conj:coeff}
Assume the notation of Theorem~\ref{thm:recursive-stuff}. If 
\[
	a_{c+1} > \max\left\{a,\frac{r_1}{z_1},\ldots,\frac{r_\elims}{z_\elims}\right\},
\]
then for $\elims<s<\elims+c$,
\begin{equation}\label{eq:conj-coeff}
y_s=c! \cdot a_1\cdots a_c \cdot \left(a_{c+1}-\sum_{i=1}^{s-n}(a_{c+1-i}-a_{i})\right).
\end{equation}
\end{conjecture}

%------------------------------------------------------------------------
%------------------------------------------------------------------------
\section{Compatibility and Stability}\label{sec:stable}

The stable behavior of the Boij-S\"oderberg decompositions of ideal powers has been studied by S.\ Mayes-Tang \cite{mayestang} when the ideal $I$ in question is homogeneous in a single degree. In particular, for $k\gg 0$ the decompositions of $I^k$ have the following properties: the number of terms in the decompositions are constant; the shapes of the pure diagrams in the decompositions are the same; the coefficients in the decompositions are given by polynomials in $k$. Given this result, D.\ Erman and S.V.\ Sam \cite{ES} ask if similar asymptotic stabilization results can be expected in other contexts. In Proposition~\ref{prop:stabilize} we show a positive answer in relation to the elimination order.  

\begin{dfn}
	The elimination order of $\bet(\lst a {c+1})$ is \defi{compatible with} the elimination order of $\bet(\lst a c)$ if $\elimorder{\bet(\lst a c)}$ is the beginning of the sequence $\elimorder{\bet(\lst a {c+1})}$. 
\end{dfn}

\begin{prop}\label{prop:stabilize}
	Let $S = \kk [\lst x n]$ and $I = (\lst f {c+1})$ be an ideal generated by a homogeneous regular sequence with $\deg(f_i) = a_i$. If the Boij-S\"oderberg decomposition of $\bet(\lst a c)$ has no instance of mass elimination, then there exists an $N > 0$ such that for all $a_{c+1} > N$,
	\begin{enumerate}
	 	\item\label{prop:stabilize1} the number of terms in the Boij-S\"oderberg decomposition of $\bet(S/I)$ is constant;
	 	\item\label{prop:stabilize2} the elimination order of $\bet(S/I)$ is compatible with that of $\bet(\lst a c)$;
	 	\item\label{prop:stabilize3} the elimination order of $\bet(S/I)$ stabilizes; 
	 	\item\label{prop:stabilize4} if $\elims$ is the number of terms in the decomposition of $\bet(\lst a c)$, then the first and last $\elims$ coefficients in the Boij-S\"oderberg decomposition of $\bet(S/I)$ are given by linear polynomials in $a_{c+1}$. 
	\end{enumerate} 
\end{prop}

\begin{proof}
	We know by Corollary~\ref{cor:algisBS} that Algorithm~\ref{NewAlgorithm} is the Boij-S\"oderberg decomposition when $a_{c+1}$ is chosen large enough. From \ref{p3} we see there are $2\elims + c -1$ terms for all choices of $a_{c+1}$, hence \eqref{prop:stabilize1} holds. Similarly, \eqref{prop:stabilize2} can be seen from \ref{p1}, and for \eqref{prop:stabilize3}, the elimination order is fixed by Algorithm~\ref{NewAlgorithm}. Finally, for \eqref{prop:stabilize4}, Theorem~\ref{thm:recursive-stuff} shows that the appropriate coefficients are linear polynomials in $a_{c+1}$.
\end{proof}

It is worth noting that Conjecture~\ref{conj:coeff} implies \emph{all} the coefficients in the Boij-S\"oderberg decomposition of $\bet(S/I)$ are given by polynomials in $a_{c+1}$. This is the case for $c+1 = 4$. 

\begin{example}\label{ex:stabilze}
Let $I=(x_1^2,x_2^3,x_3^4,x_4^{a_4})$ in the polynomial ring $S = \kk[x_1,x_2,x_3,x_4]$. Using the results from Theorem~\ref{thm:codim4}, we are able to determine the decomposition of $\beta(S/I)$ as a function of $a_4$.
\begin{align*}
	\bet(S/I) & =  (42\, a_4+294) \cdot \pure{\ee^1} + (12\, a_4-36)\cdot \pure{\ee^2} \\ 
		& + (36\, a_4+378)\cdot \pure{\ee^3} + (12\, a_4-144)\cdot \pure{\ee^4} \\
		& + (42\, a_4-204)\cdot \pure{\ee^5} + (144\, a_4-288)\cdot \pure{\ee^6} \\
		& + (144\, a_4-288)\cdot \pure{\ee^7}+ (42\, a_4-204)\cdot \pure{\ee^8} \\
		& + (12\, a_4-144)\cdot \pure{\ee^9} + (36\, a_4+378)\cdot \pure{\ee^{10}}\\ 
		& + (12\, a_4-36)\cdot \pure{\ee^{11}} + (42\, a_4+294) \cdot \pure{\ee^{12}}
\end{align*}
Observe that all the coefficients are linear in $a_4$ and the number of terms is constant (provided $a_4 \not = 12$). In this example, $a = 9$ and the maximum given in Theorem~\ref{thm:recursive-stuff} is $12$, and we give a brief analysis of the behavior of the decomposition around this bound. First, when $a_4 = 12$, observe that we have the traditional Boij-S\"oderberg decomposition with mass elimination: $y_4 = y_9 = 0$, $y_s > 0$ otherwise, and the degree sequences form a chain.  An analysis of the degree sequences will show that both compatibility and stability occur when $a_4 \geq 13$, confirming the assertions in Example~\ref{ex:EarlyCodim3DecompPart3}. When $9 \leq a_4 < 12$, we see that some of the coefficients are negative (but none are zero).  Thus, in each case, the traditional Boij-S\"oderberg decomposition uses a different chain of degree sequences.  However, given such an $a_4$, there exists a change-of-basis map from the pure diagrams used in Algorithm~\ref{NewAlgorithm} to the pure diagrams used in Algorithm~\ref{StandardDecomp}.  Indeed, both sets of pure diagrams form a basis for the support of $\beta(S/I)$.  Determining the change-of-basis map is an area for further study.  Although we do not study diagrams for which $4 \leq a_4 < 9$ much in this paper, we note Algorithm~\ref{NewAlgorithm} does provide a decomposition in these cases by Theorem~\ref{thm:Ezero}, and in these cases the coefficients are positive rational numbers (though the degree sequences do not form a chain).

\def\scl{.2}
\begin{figure}
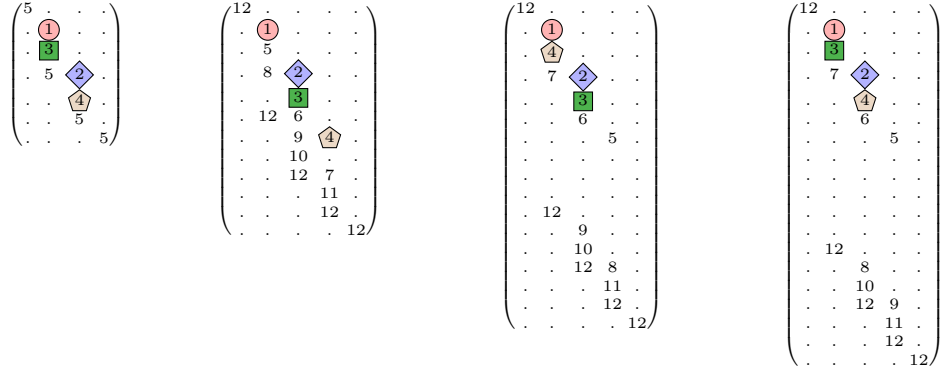

	\begin{subfigure}[t]{.11\textwidth}
	\tiny \setlength{\arraycolsep}{1pt}
	\[
		\begin{pmatrix}
		     5&\text{.}&\text{.}&\text{.}\\
		     \text{.}&\encircle{1}&\text{.}&\text{.}\\
		     \text{.}&\ensquare{3}&\text{.}&\text{.}\\
		     \text{.}&5&\enwithdiamond{2}&\text{.}\\
		     \text{.}&\text{.}&\entriangle{4}&\text{.}\\
		     \text{.}&\text{.}&5&\text{.}\\
		     \text{.}&\text{.}&\text{.}&5
		\end{pmatrix}
	\]
	\end{subfigure}\hfill
	\begin{subfigure}[t]{\scl\textwidth}
	\tiny \setlength{\arraycolsep}{1pt}
	\[
		\begin{pmatrix}
		     12&\text{.}&\text{.}&\text{.}&\text{.}\\
		     \text{.}&\encircle{1}&\text{.}&\text{.}&\text{.}\\
		     \text{.}&5&\text{.}&\text{.}&\text{.}\\
		     \text{.}&8&\enwithdiamond{2}&\text{.}&\text{.}\\
		     \text{.}&\text{.}&\ensquare{3}&\text{.}&\text{.}\\
		     \text{.}&12&6&\text{.}&\text{.}\\
		     \text{.}&\text{.}&9&\entriangle{4}&\text{.}\\
		     \text{.}&\text{.}&10&\text{.}&\text{.}\\
		     \text{.}&\text{.}&12&7&\text{.}\\
		     \text{.}&\text{.}&\text{.}&11&\text{.}\\
		     \text{.}&\text{.}&\text{.}&12&\text{.}\\
		     \text{.}&\text{.}&\text{.}&\text{.}&12
		\end{pmatrix}
	\] 
	\end{subfigure}\hfill
	\begin{subfigure}[t]{\scl\textwidth}
	\tiny \setlength{\arraycolsep}{1pt}
	\[
		\begin{pmatrix}
		     12&\text{.}&\text{.}&\text{.}&\text{.}\\
		     \text{.}&\encircle{1}&\text{.}&\text{.}&\text{.}\\
		     \text{.}&\entriangle{4}&\text{.}&\text{.}&\text{.}\\
		     \text{.}&7&\enwithdiamond{2}&\text{.}&\text{.}\\
		     \text{.}&\text{.}&\ensquare{3}&\text{.}&\text{.}\\
		     \text{.}&\text{.}&6&\text{.}&\text{.}\\
		     \text{.}&\text{.}&\text{.}&5&\text{.}\\
		     \text{.}&\text{.}&\text{.}&\text{.}&\text{.}\\
		     \text{.}&\text{.}&\text{.}&\text{.}&\text{.}\\
		     \text{.}&\text{.}&\text{.}&\text{.}&\text{.}\\
		     \text{.}&12&\text{.}&\text{.}&\text{.}\\
		     \text{.}&\text{.}&9&\text{.}&\text{.}\\
		     \text{.}&\text{.}&10&\text{.}&\text{.}\\
		     \text{.}&\text{.}&12&8&\text{.}\\
		     \text{.}&\text{.}&\text{.}&11&\text{.}\\
		     \text{.}&\text{.}&\text{.}&12&\text{.}\\
		     \text{.}&\text{.}&\text{.}&\text{.}&12
		\end{pmatrix}
	\]
	\end{subfigure}\hfill
	\begin{subfigure}[t]{\scl\textwidth}
	\tiny \setlength{\arraycolsep}{1pt}
	\[
		\begin{pmatrix}
		     12&\text{.}&\text{.}&\text{.}&\text{.}\\
		     \text{.}&\encircle{1}&\text{.}&\text{.}&\text{.}\\
		     \text{.}&\ensquare{3}&\text{.}&\text{.}&\text{.}\\
		     \text{.}&7&\enwithdiamond{2}&\text{.}&\text{.}\\
		     \text{.}&\text{.}&\entriangle{4}&\text{.}&\text{.}\\
		     \text{.}&\text{.}&6&\text{.}&\text{.}\\
		     \text{.}&\text{.}&\text{.}&5&\text{.}\\
		     \text{.}&\text{.}&\text{.}&\text{.}&\text{.}\\
		     \text{.}&\text{.}&\text{.}&\text{.}&\text{.}\\
		     \text{.}&\text{.}&\text{.}&\text{.}&\text{.}\\
		     \text{.}&\text{.}&\text{.}&\text{.}&\text{.}\\
		     \text{.}&\text{.}&\text{.}&\text{.}&\text{.}\\
		     \text{.}&12&\text{.}&\text{.}&\text{.}\\
		     \text{.}&\text{.}&8&\text{.}&\text{.}\\
		     \text{.}&\text{.}&10&\text{.}&\text{.}\\
		     \text{.}&\text{.}&12&9&\text{.}\\
		     \text{.}&\text{.}&\text{.}&11&\text{.}\\
		     \text{.}&\text{.}&\text{.}&12&\text{.}\\
		     \text{.}&\text{.}&\text{.}&\text{.}&12
	     \end{pmatrix}
	\]
	\end{subfigure}
	\caption{Elimination tables for $\bet(2,3,4)$, $\bet(2,3,4,6)$, $\bet(2,3,4,11)$, and $\bet(2,3,4,13)$ respectively.}
	\label{fig:stabilization}
\end{figure}

\end{example}

In codimension three, the structure theorem in \cite{GJMRSW} shows there are exactly two non-trivial elimination orders for the diagram $\bet(a_1,a_2,a_3)$. In particular, the case when $a_1=a_2 < a_3$ has a different elimination order than all other cases. In codimension four, even if we avoid mass elimination, there are still plenty of different elimination orders as detailed in \cite{GJMRSW}. However according to Proposition~\ref{prop:stabilize}, if we avoid mass elimination and $a_{c+1} \gg 0$, then the elimination order stabilizes. When mass elimination does occur, compatibility fails when decomposing $\bet(4,5,7,9,a_5)$ with large values of $a_5$, but the other results seem to hold. This begs the following open question: When does a complete intersection of codimension $c$ have mass elimination, and why does it affect the stability of the related codimension $c+1$ diagram?

%------------------------------------------------------------------------
%------------------------------------------------------------------------
\section{Case Study: Codimension Four}\label{sec:codim4}

When $c = 3$, we describe the remainders relative to $a_4$ when it satisfies the hypotheses in Theorem~\ref{thm:recursive-stuff}, which then allows us to find a closed formula for bound on $a_4$ in terms of $a_1$, $a_2$, and $a_3$.  Recall that no mass elimination occurs in codimension three, thus trivially satisfying that hypothesis of Theorem~\ref{thm:recursive-stuff}.

\begin{lem}\label{lem:codim4r}
	The remainders of $\bet(a_1,a_2,a_3,a_4)$ with respect to $a_4$ are given by the following when $a_1 < a_2 < a_3$:
	\begin{align*}
		r_1 &= -a_1 a_2 (a_2+a_3)^2, \\
		r_2 &=  a_1 a_2 (a_1 a_2+2 a_1 a_3-a_3^2), \\
		r_3 &= -a_1 a_2 (a_1-a_2+a_3) (a_1+a_2+4 a_3), \\
		r_4 &=  a_1 a_2 (a_1^2+4 a_1 a_3-a_2 a_3), \\
		r_5 &= -a_1^2 a_2 (a_2-5 a_3).
\intertext{
	When $a_1 < a_2 = a_3$ we have:
}
		r_1' &= r_1\big|_{a_2 = a_3} =  -4 a_1 a_2^3,\\ 
		r_2' &= (r_2+r_3+r_4)\big|_{a_2 = a_3} = 2 a_1^2 a_2^2-2 a_1 a_2^3, \\ 
		r_3' &= r_5\big|_{a_2 = a_3} = 4 a_1^2 a_2^2.
\intertext{	
	When $a_1 = a_2 < a_3$ we have:
}
		r_1'' &= -2a_1^2a_3^2, \\ % != r_1 + r_2 b/c of different elim order
		r_2'' &= -2a_1^3 a_3 - 4a_1^2 a_3^2, \\ % != r_3 b/c of different elim order
		r_3'' &= 8a_1^3 a_3. % != r_4 + r_5 b/c of different elim order
\intertext{
	Finally, when $a_1 = a_2 = a_3$, we have:
}
		r_1''' &= (r_1+r_2+r_3+r_4+r_5)\big|_{a_1 = a_2 = a_3} = 0. 
	\end{align*}

\end{lem}
\begin{proof}
	Recall from Theorem~\ref{thm:recursive-stuff} that $r_k$ is defined recursively in terms of the previous $r_s$. In the case of codimension four we have $r_1 = (j_1 -a)z_1$ and
	\[
		r_k = \left(\frac{j_k-a}{p_k}b_k-\sum_{s = 1}^{k-1}\frac{\pure{\dd^s}_{i_k,j_k}}{p_k}r_s\right)
	\] 
	for $2 \leq k \leq 5$, where $a = a_1+a_2+a_3$ and $b_k, i_k, j_k, p_k$ are described in Remark~\ref{rmk:notation}. From \cite{GJMRSW} we have the degree sequences and coefficients from the decomposition of the Betti diagram of a codimension three complete intersection,
	\begin{align*}
		\dd^1 &= (0,a_1,a_1+a_2,a_1+a_2+a_3), & z_1 &= a_1a_2(a_2+a_3),\\
		\dd^2 &= (0,a_2,a_1+a_2,a_1+a_2+a_3), & z_2 &= a_1a_2(a_3-a_1), \\
		\dd^3 &= (0,a_2,a_1+a_3,a_1+a_2+a_3), & z_3 &= 2a_1a_2(a_1+a_3-a_2), \\
		\dd^4 &= (0,a_3,a_1+a_3,a_1+a_2+a_3), & z_2 &= a_1a_2(a_3-a_1), \\
		\dd^5 &= (0,a_3,a_2+a_3,a_1+a_2+a_3), & z_5 &= a_1a_2(a_2+a_3).
	\end{align*}
	Notice the elimination order in codimension three is not fixed. When $a_1 < a_2 < a_3$, we have the elimination order
	\begin{align*}
		i_1 &= 1, & j_1 &= a_1, \\
		i_2 &= 2, & j_2 &= a_1 + a_2, \\
		i_3 &= 1, & j_3 &= a_2, \\
		i_4 &= 2, & j_4 &= a_1 + a_3, \\
		i_5 &= 3, & j_5 &= a_1+a_2+a_3.
	\end{align*}

	When $a_1 < a_2 = a_3$, the degree sequences above collapse into three degree sequences
	\begin{align*}
		{\dd^1}' &= (0, a_1, a_1+a_2, a_1 + 2a_2), \\
		{\dd^2}' &= (0, a_2, a_1 + a_2, a_1 + 2a_2), \\
		{\dd^3}' &= (0,a_2, 2a_2, a_1 + 2a_2), 
	\end{align*}
	with elimination order:
	\[ 
		i_1 = 1, j_1 = a_1; \quad i_2 = 2, j_2 = a_1 + a_2; \quad i_3 = 3, j_3 = a_1 + 2a_2. 
	\]	
	These are consistent with the elimination order for $a_1 < a_2 < a_3$.

	When $a_1 = a_2 = a_3$, we are in the special case where $\elims = 1$, $${\dd^1}''' = (0, a_1, 2a_1, 3a_1),$$ and $i_1 = 3, j_3 = 3a_1$.  This is trivially consistent with the elimination order for $a_1 < a_2 < a_3$.

	However, when $a_1 = a_2 < a_3$, the degree sequences above collapse into three degree sequences
	\begin{align*} 
		{\dd^1}'' &= (0, a_1, 2a_1, 2a_1 + a_3), \\
		{\dd^2}'' &= (0,a_1, a_1+ a_3,2a_1 + a_3), \\
		{\dd^3}'' &= (0, a_3, a_1 + a_3, 2a_1 + a_3),
	\end{align*}
	with an elimination order that begins in column 2 instead of column 1:
	\[i_1 = 2, j_1 = 2a_1; \quad i_2 = 1, j_2 = a_1; \quad i_3 = 3, j_3 = 2a_1  + a_3. \]

	We proceed via two cases: when the elimination order is consistent with the general case $a_1 < a_2 < a_3$ and when it isn't, i.e., when $a_1 = a_2 < a_3$.

	Using the recursive formula, each $r_s$ can be explicitly written. The zeros below represent the instances in the sum where the respective degree sequence contained a zero in the position $i_s,j_s$. 
	\begin{align*}
		r_1 &= (j_1 - a)z_1, \\
		r_2 &= \frac{(j_2 - a)}{p_2}b_2 - \frac{\pure{\dd^1}_{i_2,j_2}}{p_2}r_1, \\
		r_3 &= \frac{(j_3 - a)}{p_3}b_3 \makebox[0pt][l]{ $-$ 0 \ }\phantom{- \frac{\pure{\dd^1}_{i_3,j_3}}{p_3}r_1} - \frac{\pure{\dd^2}_{i_3,j_3}}{p_3}r_2, \\
		r_4 &= \frac{(j_4 - a)}{p_4}b_4 \makebox[0pt][l]{ $-$ 0 \ }\phantom{- \frac{\pure{\dd^1}_{i_4,j_4}}{p_4}r_1} \makebox[0pt][l]{ $-$ 0 \ }\phantom{- \frac{\pure{\dd^2}_{i_4,j_4}}{p_4}r_2} - \frac{\pure{\dd^3}_{i_4,j_4}}{p_4}r_3, \\
		r_5 &= \frac{(j_5 - a)}{p_5}b_5 - \frac{\pure{\dd^1}_{i_5,j_5}}{p_5}r_1 - \frac{\pure{\dd^2}_{i_5,j_5}}{p_5}r_2 - \frac{\pure{\dd^3}_{i_5,j_5}}{p_5}r_3 - \frac{\pure{\dd^4}_{i_5,j_5}}{p_5}r_4. 
	\end{align*}
	Applying the appropriate recursive substitutions results in the desired formulas.

	When $a_1 < a_2 = a_3$ we have $\dd^2 = \dd^3 = \dd^4$ and $\elims(\bet^3)=3$ as described above.
	From this we have the following formulas for the remainders.
	\begin{align*}
		r_1' &= (j_1 - a)z_1, \\
		r_2' &= \frac{(j_2 - a)}{p_2}\cdot 2 - \frac{\pure{{\dd^1}'}_{i_2,j_2}}{p_2}r_1, \\
		r_3' &= \frac{(j_3 - a)}{p_3} - \frac{\pure{{\dd^1}'}_{i_3,j_3}}{p_3}r_1 - \frac{\pure{{\dd^2}'}_{i_3,j_3}}{p_5}r_2.
	\end{align*}
	Making the appropriate substitutions we obtain the desired formulas.  When $a_1 = a_2 = a_3$, similar rote manipulations yield that $r'''_1 = 0$.

	For the case when $a_1 = a_2 < a_3$ notice that $\dd^1 = \dd^2$ and $\dd^4 = \dd^5$ above. Once again we have $\elims(\bet^3) = 3$ but the elimination order is different than the previous two cases; in this case we start eliminating in the first column instead of the second column.  The details of the proof are similar to the previous cases.
\end{proof}

In order to simplify the next result, we define the following ratios between the remainders in Lemma~\ref{lem:codim4r} and the coefficients of the codimension three decomposition where $a_1 < a_2 < a_3$ \cite{GJMRSW} as 
\newcommand{\rat}[1]{\left({\frac{r}{z}}\right)_{#1}}
\begin{align*}
	\rat 1 &= -(a_2+a_3), \\ 
	\rat 2 &= \frac{a_1(a_2+2a_3)-a_3^2}{a_3 - a_1}, \\ 
	\rat 3 &= \frac{a_3(a_2+a_3) - a_1(a_1+3a_3)}{2a_1}, \\
	\rat 4 &= \frac{a_1(a_1 + 4a_3) - a_2(a_1+2a_3)}{(a_3-a_1)}, \\
	\rat 5 &= \frac{ a_2(5a_1a_3 + a_2^2)-({a}_1 {a}_2^2+{a}_1^2 {a}_3+{a}_1 {a}_3^2+{a}_2 {a}_3^2)}{a_2(a_2+a_3)}.\\
\intertext{Similarly, when $a_1 < a_2 = a_3$,}
\rat 1' &= -2a_2,\\ %negative
\rat 2' &= \frac{1}{2}(a_1-a_2),\\ %negative
\rat 3' &= 2a_1. \\ %positive, ding ding ding
\intertext{When $a_1 = a_2 < a_3$, }
\rat 1'' &= -a_3,\\
\rat 2'' &= -(a_1+2a_3),\\
\rat 3'' &= 4a_1.\\
\intertext{Finally, when $a_1 = a_2 = a_3$,}
\rat 1''' &= 0.
\end{align*}

Using the above notation, we state the following partial classification theorem for the \BS decomposition of complete intersections in codimension four. 
	
\begin{thm}\label{thm:codim4}
	Let $S = \kk [x_1,x_2,x_3,x_4]$ and $I = (f_1,f_2,f_3,f_4)$ be an ideal generated by a homogeneous regular sequence with $\deg(f_i) = a_i$, where $a_i < a_{i+1}$ for all $i$.	If $a_4 \gg 0$, then the decomposition of $\bet(S/I)$ obtained from Algorithm~\ref{StandardDecomp} is completely determined by the degrees $a_1,a_2,a_3,a_4$.  In particular, we have the following decompositions broken down by cases. 

	Case 1: $a_1 < a_2 \leq a_3$ or $a_1 = a_2 = a_3$; when 
	\[
		a_4 > \max\left\{a_1+a_2+a_3,\rat 1, \rat 2, \rat 3, \rat 4, \rat 5\right\}
	\]
	(note that if $a_1 \leq a_2 = a_3$, this maximum is $a_1 +2a_2$), then:
	\begin{align*}
		\bet(S/I)
			&= a_1a_2 (a_2+a_3) (a_2+a_3+a_4) \cdot \pure{\ee^1}\\
			&- a_1a_2 (a_1 a_2+2 a_1 a_3-a_3^2+a_1 a_4-a_3 a_4)\cdot \pure{\ee^2} \\
			&+ a_1a_2 (a_1-a_2+a_3) (a_1+a_2+4 a_3+2 a_4) \cdot \pure{\ee^3} \\
			&- a_1a_2 (a_1^2+4 a_1 a_3-a_2 a_3+a_1 a_4-a_3 a_4) \cdot \pure{\ee^4}\\
			&+ a_1a_2 (a_1 a_2-5 a_1 a_3+a_2 a_4+a_3 a_4) \cdot \pure{\ee^5} \\
			&+ 6a_1a_2a_3 (a_1-a_3+a_4) \cdot \pure{\ee^6}\\
			&+ 6a_1a_2a_3 (a_1-a_3+a_4) \cdot \pure{\ee^7}\\			
			&+ a_1a_2 (a_1 a_2-5 a_1 a_3+a_2 a_4+a_3 a_4) \cdot \pure{\ee^8} \\			
			&- a_1a_2 (a_1^2+4 a_1 a_3-a_2 a_3+a_1 a_4-a_3 a_4) \cdot \pure{\ee^9}\\
			&+ a_1a_2 (a_1-a_2+a_3) (a_1+a_2+4 a_3+2 a_4) \cdot \pure{\ee^{10}} \\			
			&- a_1a_2 (a_1 a_2+2 a_1 a_3-a_3^2+a_1 a_4-a_3 a_4)\cdot \pure{\ee^{11}} \\
			&+ a_1a_2 (a_2+a_3) (a_2+a_3+a_4) \cdot \pure{\ee^{12}}
	\intertext{where the degree sequences $\ee^s$ are given by}
		\ee^1 &= (0, a_1, a_1+a_2, a_1+a_2+a_3, a_1+a_2+a_3+a_4), \\
		\ee^2 &= (0, a_2, a_1+a_2, a_1+a_2+a_3, a_1+a_2+a_3+a_4), \\
		\ee^3 &= (0, a_2, a_1+a_3, a_1+a_2+a_3, a_1+a_2+a_3+a_4), \\
		\ee^4 &= (0, a_3, a_1+a_3, a_1+a_2+a_3, a_1+a_2+a_3+a_4), \\
		\ee^5 &= (0, a_3, a_2+a_3, a_1+a_2+a_3, a_1+a_2+a_3+a_4), \\
		\ee^6 &= (0, a_3, a_2+a_3, a_1+a_2+a_4, a_1+a_2+a_3+a_4), \\
		\ee^7 &= (0, a_3, a_1+a_4, a_1+a_2+a_4, a_1+a_2+a_3+a_4), \\		
		\ee^8 &= (0, a_4, a_1+a_4, a_1+a_2+a_4, a_1+a_2+a_3+a_4), \\			
		\ee^9 &= (0, a_4, a_2+a_4, a_1+a_2+a_4, a_1+a_2+a_3+a_4), \\					
		\ee^{10} &= (0, a_4, a_2+a_4, a_1+a_3+a_4, a_1+a_2+a_3+a_4), \\							
		\ee^{11} &= (0, a_4, a_3+a_4, a_1+a_3+a_4, a_1+a_2+a_3+a_4), \\	
		\ee^{12} &= (0, a_4, a_3+a_4, a_2+a_3+a_4, a_1+a_2+a_3+a_4).
	\intertext{Case 2: $a_1 = a_2 < a_3$ and $a_4 \geq \max\{2a_1 + a_3,4a_1\}$,
    we have
	} %%FACTORed
	\bet(S/I)
	&= 2a_1^2a_3(a_3+a_4) \cdot \pure{\ee^1} \\
	&+ 2a_1^2a_3(a_1+2 a_3+a_4)\cdot \pure{\ee^2} \\
	&- 2a_1^2a_3(4 a_1-a_4) \cdot \pure{\ee^3} \\
	&+ 6a_1^2a_3(a_1-a_3+a_4) \cdot \pure{\ee^4} \\
	&+ 6a_1^2a_3(a_1-a_3+a_4) \cdot \pure{\ee^5} \\
	&- 2a_1^2a_3(4a_1-a_4) \cdot \pure{\ee^6} \\
	&+ 2a_1^2a_3(a_1+2 {a}_3+{a}_4) \cdot \pure{\ee^7} \\
	&+ 2a_1^2a_3(a_3+a_4) \cdot \pure{\ee^8}
	\intertext{
	where the degree sequenes $\ee^s$ are given by
	}
	\ee^1 &= (0,	a_1,	2a_1,	 2a_1+a_3 ,		2a_1+a_3+a_4),\\
	\ee^2 &= (0,	a_1,	a_1+a_3, 2a_1+a_3, 		2a_1+a_3+a_4),\\
	\ee^3 &= (0,	a_3,	a_1+a_3, 2a_1+a_3, 		2a_1+a_3+a_4),\\
	\ee^4 &= (0,	a_3,	a_1+a_3, 2a_1+a_4, 		2a_1+a_3+a_4),\\
	\ee^5 &= (0,	a_3,	a_1+a_4, 2a_1+a_4, 		2a_1+a_3+a_4),\\
	\ee^6 &= (0,	a_4,	a_1+a_4, 2a_1+a_4, 		2a_1+a_3+a_4),\\
	\ee^7 &= (0,	a_4,	a_1+a_4, a_1+a_3+a_4, 	2a_1+a_3+a_4),\\
	\ee^8 &= (0,	a_4,	a_3+a_4, a_1+a_3+a_4, 	2a_1+a_3+a_4).
	\end{align*}

\end{thm}
\begin{proof}
	Before showing that these are the correct decompositions, we observe that when $a_1 < a_2 = a_3$, $\max\left\{a_1+2a_2,\rat 1',\rat 2', \rat 3'\right\} = a_1+2a_2$;
	when $a_1 = a_2 < a_3$, $\max\left\{2a_1+a_3,\rat 1'',\rat 2'',\rat 3''\right\} = \max\left\{2a_1+a_3,4a_1\right\};$
	and when $a_1 = a_2 = a_3$, $\max\left\{3a_1,\rat 1'''\right\} = 3a_1$.  Thus, in each case, $a_4$ satisfies the hypotheses of Theorem~\ref{thm:recursive-stuff}.  This means that the elimination order has stabilized and Algorithm~\ref{NewAlgorithm} aligns with Algorithm~\ref{StandardDecomp}. 

	To calculate the degree sequences $\ee^s$, we use the degree sequences from the decomposition of $\bet(S/(f_1,f_2,f_3))$ found in \cite{GJMRSW},
	\begin{align*}
		\dd^1 &= (0, a_1, a_1+a_2, a_1+a_2+a_3), \\
		\dd^2 &= (0, a_2, a_1+a_2, a_1+a_2+a_3), \\
		\dd^3 &= (0, a_2, a_1+a_3, a_1+a_2+a_3), \\
		\dd^4 &= (0, a_3, a_1+a_3, a_1+a_2+a_3), \\
		\dd^5 &= (0, a_3, a_2+a_3, a_1+a_2+a_3).		
	\end{align*}
	These degree sequences are what defines the two cases above. In particular, Case 1 and Case 2 are derived from the different elimination orders of codimension three complete intersections. Focusing on Case 1, notice when $a_1 < a_2 < a_3$, \ref{p1} of Algorithm~\ref{NewAlgorithm} produces the desired $\ee^s$ for $1\ls s \ls 5$, while \ref{p2} produces $\ee^6$. 
	By duality we see that $\ee^s = (\ee^{13-s})^\vee$ for $7 \ls s \ls 12$. Similar observations produce the results for the remaining instances of Case 1.

	Because the coefficients are symmetric, we only need to calculate the first six to show Case 1. By Theorem \ref{thm:recursive-stuff} we know that the coefficients from \ref{p1} are positive and are a function of the degrees. In particular, for $1 \ls s \ls 5$:
	\allowdisplaybreaks % allows for equations on multiple pages
	\begin{align*}		
		y_1 &= a_1a_2(a_2+a_3)\cdot a_4 - r_1 \\
			&= a_1a_2(a_2+a_3)\cdot a_4 - (-a_1a_2({a_2+a_3})^2) \\
			&= a_1a_2 (a_2+a_3) (a_2+a_3+a_4); \\ \\
		y_2 &=a_1a_2(a_3-a_1)\cdot a_4 - r_2 \\
			&=a_1a_2(a_3-a_1)\cdot a_4 - (a_1a_2 (a_1 a_2+2 a_1 a_3-a_3^2))	\\
			&= - a_1a_2 (a_1 a_2+2 a_1 a_3-a_3^2+a_1 a_4-a_3 a_4); \\ \\
		y_3 &= 2a_1a_2(a_1+a_3-a_2)\cdot a_4 - r_3 \\
			&= 2a_1a_2(a_1+a_3-a_2)\cdot a_4 - (-a_1a_2 (a_1-a_2+a_3) (a_1+a_2+4 a_3)) \\
			&= a_1a_2 (a_1-a_2+a_3) (a_1+a_2+4 a_3+2 a_4); \\ \\
		y_4 &= a_1a_2(a_3-a_1)\cdot a_4 - r_4 \\
			&= a_1a_2(a_3-a_1)\cdot a_4 - (a_1a_2 (a_1^2+4 a_1 a_3-a_2 a_3)) \\
			&= - a_1a_2 (a_1^2+4 a_1 a_3-a_2 a_3+a_1 a_4-a_3 a_4); \\ \\
		y_5 &= a_1a_2(a_2+a_3)\cdot a_4 - r_5 \\
			&= a_1a_2(a_2+a_3)\cdot a_4 - (- a_1^2a_2(a_2-5 a_3)) \\
			&= a_1a_2 (a_1 a_2-5 a_1 a_3+a_2 a_4+a_3 a_4).
	\end{align*}
	\allowdisplaybreaks[0] % turns display breaks off

	To complete the proof of Case 1, we only need to show the coefficient $y_6$ aligns with Conjecture~\ref{conj:coeff}. According to the elimination order, $y_6\cdot \pure{\ee^6}$ targets the entry of $\bet_4$ in the third column with degree $a_2+a_3$. As $a_1 < a_2 <a_3$, only two degree sequences in the chain contribute to this entry, $\ee^5$ and $\ee^6$. 

	The desired entries $\pure{\ee^5}_{2,a_2+a_3}$ and $\pure{\ee^6}_{2,a_2+a_3}$ are
	\[
		\frac1{a_2a_1(a_2+a_3)(a_1+a_4)} \text{ and } \frac1{a_2(a_2+a_3)(a_1-a_3+a_4)(a_1+a_4)}		
	\]
	respectively. Summing these quantities $y_5\cdot\pure{\ee^5}_{2,a_2+a_3} + y_6\cdot\pure{\ee^6}_{2,a_2+a_3}$ gives
	\[		
		 \frac{a_1a_2 (a_1 a_2-5 a_1 a_3+a_2 a_4+a_3 a_4)}{a_2a_1(a_2+a_3)(a_1+a_4)} + \frac{6a_1a_2a_3 (a_1-a_3+a_4)}{a_2(a_2+a_3)(a_1-a_3+a_4)(a_1+a_4)} = 1,
	\]
	the desired result. 

	When $a_1 < a_2 = a_3$, ${\ee^3}' = \ee^5$, ${\ee^4}'=\ee^6$ are the only relevant sequences. As such, the above calculations produce the desired result. For the remaining instance of Case 1, $a_1 = a_2 = a_3$, we have $\elims = 1$ and hence 
	\begin{align*}
		{\ee^1}'&=\ee^1=\ee^2=\ee^3=\ee^4=\ee^5; \\
		{\ee^2}'&=\ee^6
	\end{align*}
	are the only contributing degree sequences to the position $2,a_2+a_3$. The respective coefficients to $\pure{{\ee^1}'}$ is 
	\[
		y_1+y_1+y_3+y_4+y_5=6a_1^3a_4.
	\]
	Summing $6a_1^3a_4\cdot\pure{{\ee^1}'}_{2,a_2+a_3}$ and $y_6 \cdot \pure{{\ee^2}'}_{2,a_2+a_3}$ gives
	\[
		\frac{6a_1^3a_4}{2a_1^3(a_1+a_4)} + \frac{6a_1a_2a_3 (a_1-a_3+a_4)}{2a_1^2a_4(a_1+a_4)}=3,
	\]
	the desired result.

	Similar computations complete Case 2.
\end{proof}

	Combining the results from \cite{GJMRSW} with the above theorem yields the following.

\begin{cor}\label{cor:conj}
	Conjecture \ref{conj:coeff} holds in codimension $c \ls 3$. 
\end{cor}

%\bibliography{bib}{}

\begin{thebibliography}{GJM{\etalchar{+}}15}

\bibitem[AGHS17]{reu}
Michael~T. Annunziata, Courtney~R. Gibbons, Cole Hawkins, and Alexander~J.
  Sutherland, \emph{Rational combinations of betti diagrams of complete
  intersections}, Journal of Algebra and Its Applications \textbf{to appear}
  (2017).

\bibitem[BS08]{BS1}
Mats Boij and Jonas S{\"o}derberg, \emph{Graded {B}etti numbers of
  {C}ohen-{M}acaulay modules and the multiplicity conjecture}, J. Lond. Math.
  Soc. (2) \textbf{78} (2008), no.~1, 85--106. \MR{2427053 (2009g:13018)}

\bibitem[BS12]{BS2}
\bysame, \emph{Betti numbers of graded modules and the multiplicity conjecture
  in the non-{C}ohen-{M}acaulay case}, Algebra Number Theory \textbf{6} (2012),
  no.~3, 437--454. \MR{2966705}

\bibitem[EKKS15]{EKS}
Sabine El~Khoury, Manoj Kummini, and Hema Srinivasan, \emph{Bounds for the
  multiplicity of {G}orenstein algebras}, Proc. Amer. Math. Soc. \textbf{143}
  (2015), no.~1, 121--128. \MR{3272737}

\bibitem[Erm10]{ermanHorrocks}
Daniel Erman, \emph{A special case of the {B}uchsbaum-{E}isenbud-{H}orrocks
  rank conjecture}, Math. Res. Lett. \textbf{17} (2010), no.~6, 1079--1089.
  \MR{2729632}

\bibitem[ES09]{ES1}
David Eisenbud and Frank-Olaf Schreyer, \emph{Betti numbers of graded modules
  and cohomology of vector bundles}, J. Amer. Math. Soc. \textbf{22} (2009),
  no.~3, 859--888. \MR{2505303 (2011a:13024)}

\bibitem[ES16]{ES}
Daniel Erman and Steven~V Sam, \emph{Questions about {B}oij-{S\"o}derberg
  theory}, Algebraic geometry, bootcamp volume, Proc. Sympos. Pure Math., To
  appear, 2016.

\bibitem[Fl{\o}12]{floystad-expos}
Gunnar Fl{\o}ystad, \emph{Boij-{S}\"oderberg theory: introduction and survey},
  Progress in commutative algebra 1, de Gruyter, Berlin, 2012, pp.~1--54.
  \MR{2932580}

\bibitem[FLS16]{samGrass}
Nicolas Ford, Jake Levenson, and Steven~V Sam, \emph{Towards
  {B}oij-{S\"o}derberg theory for {G}rassmannians: the case of square
  matrices}, arXiv:1608.04058v1 (2016).

\bibitem[GJM{\etalchar{+}}15]{GJMRSW}
Courtney Gibbons, Jack Jeffries, Sarah Mayes, Claudiu Raicu, Branden Stone, and
  Bryan White, \emph{Non-simplicial decompositions of {B}etti diagrams of
  complete intersections}, J. Commut. Algebra \textbf{7} (2015), no.~2,
  189--206. \MR{3370483}

\bibitem[GS]{M2}
Daniel~R. Grayson and Michael~E. Stillman, \emph{Macaulay2, a software system
  for research in algebraic geometry}, Available at
  \href{http://www.math.uiuc.edu/Macaulay2/}%
  {http://www.math.uiuc.edu/Macaulay2/}.

\bibitem[HK84]{HK}
J.~Herzog and M.~K\"uhl, \emph{On the {B}etti numbers of finite pure and linear
  resolutions}, Comm. Algebra \textbf{12} (1984), no.~13-14, 1627--1646.
  \MR{743307}

\bibitem[HS98]{HS}
J{\"u}rgen Herzog and Hema Srinivasan, \emph{Bounds for multiplicities}, Trans.
  Amer. Math. Soc. \textbf{350} (1998), no.~7, 2879--2902. \MR{1458304
  (99g:13033)}

\bibitem[McC12]{McC}
Jason McCullough, \emph{A polynomial bound on the regularity of an ideal in
  terms of half of the syzygies}, Math. Res. Lett. \textbf{19} (2012), no.~3,
  555--565. \MR{2998139}

\bibitem[MT15]{mayestang}
Sarah Mayes-Tang, \emph{Stabilization of {B}oij-{S\"o}derberg decompositions of
  ideal powers}, arXiv:1509.08544v1 (2015).

\bibitem[NS13]{NS}
Uwe Nagel and Stephen Sturgeon, \emph{Combinatorial interpretations of some
  {B}oij--{S}\"oderberg decompositions}, J. Algebra \textbf{381} (2013),
  54--72. \MR{3030509}

\bibitem[Whi14]{whieldon}
Gwyneth Whieldon, \emph{Stabilization of {B}etti tables}, J. Commut. Algebra
  \textbf{6} (2014), no.~1, 113--126. \MR{3215565}

\end{thebibliography}
%\bibliographystyle{amsalpha}

\newcommand{\etalchar}[1]{$^{#1}$}
\providecommand{\bysame}{\leavevmode\hbox to3em{\hrulefill}\thinspace}
\providecommand{\MR}{\relax\ifhmode\unskip\space\fi MR }
\providecommand{\MRhref}[2]{%
  \href{http://www.ams.org/mathscinet-getitem?mr=#1}{#2}
}
\providecommand{\href}[2]{#2}

\end{document}